\documentclass[oneside,reqno,11pt]{amsart}
\usepackage{epsfig}
\usepackage{color}
\usepackage{amssymb,amsmath,amsthm,amstext,amsfonts}
\usepackage{color}
\usepackage{graphicx}
\newtheorem{theorem}{Theorem}[section]
\newtheorem{corollary}[theorem]{Corollary}
\newtheorem{proposition}[theorem]{Proposition}
\newtheorem{lemma}[theorem]{Lemma}

\theoremstyle{definition}

\newtheorem{definition}[theorem]{Definition}

\newtheorem{example}[theorem]{Example}

\newtheorem{remark}[theorem]{Remark}

\newcommand{\ZZ}{\mathbb Z}

\newcommand{\topp}{\operatorname{top}}

\newcommand{\Ptop}{P_{\topp}}

\newcommand{\cE}{\mathcal{E}}

\newcounter{main}

\makeatletter
\let\c@equation\c@theorem
\makeatother
\numberwithin{equation}{section}

\title[Variational principles for   nonautonomous dynamical systems]{ Variational principles for  nonautonomous dynamical systems}

\author[A. Bi\'s]{Andrzej Bi\'s}
\address{Faculty of Mathematics and Computer Science, {\L}\'od\'z University,  {\L}\'od\'z, Poland}
\email{andrzej.bis@wmii.uni.lodz.pl}


\date{\today}

\begin{document}

\begin{abstract}
We study thermodynamical formalism of a discrete nonautonomous dynamical system determined by a sequence of continuous self-maps
of a compact metric space. Using the methods of Convex Analysis we get variational principles for pressure functions determined by the system.

\end{abstract}

\maketitle


\section{Introduction and structure of the paper}\label{se:introd}
In 1959, Kolmogorov and Sinai developed the metric
entropy from Shannon's information theory into ergodic
theory. Since then, the notion of entropy has played a crucial 
role in the study of dynamical systems.
In 1965, the notion of topological entropy, denoted here by $h_{top}(f),$ of a classical dynamical system, determined by a single continuous self-map $f:X \to X$ of a compact topological space $X,$ was introduced by Adler, Konheim and McAndrew  \cite{AKMc}
 as a topological conjugacy invariant. 
 In 1971, Bowen \cite{Bow1}  and Dinaburg \cite{Din} gave a new, but equivalent, definition in the case when the space under consideration is metrizable. 
 
 Another approach which resembles the construction of Hausdorff measure and Hausdorff dimension was used by Bowen \cite{Bow2} to extend the concept of topological entropy for 
noncompact sets. This approach was developed by Pesin and Pitskel \cite{PP} and presented in details by Pesin \cite{Pes} who elaborated a dimension theory of Carath\'eodory structures.

 Due to Krylov-Bogoliubov Theorem  \cite{BK} for a continuous map $f:X \to X$ there exists a Borel probability $f-$invariant measure $\mu$, defined on a Borel $\sigma-$algebra on $X.$ Therefore, the space of $\mathcal P_{f}(X)$  of $f$-invariant probability measures defined on the $\sigma$-algebra $\mathfrak B(X)$ of the Borel subsets of $X$ is not empty. Thus,
  one can define and study measure-theoretic entropy, denoted  here  by $h_{\mu}(f),$ of  a continuous transformation $f$ with respect to a measure $\mu \in \mathcal P_{f}(X).$ The famous variational principle relates topological
  entropy of a continuous self-map $f:X \to X,$ acting on a compact metric space $X,$ to its measure-theoretic entropy. It says that
  $$h_{top}(f)=\sup \{h_{\mu}(f): \mu \in \mathcal P_{f}(X)\}.$$

  Equivalently, it can be rewritten for the topological pressure $\Ptop(f,\varphi)$ of the map $f$  (Theorem 9.10 in \cite{Wa})      as
  \begin{equation}\label{e.VP}
\Ptop(f,\varphi)= \sup_{\mu\,\in\, \mathcal P_{f}(X)}\,\,\Big\{h_\mu(f) + \int_X \varphi\, d\mu\Big\},
\end{equation}
where $\varphi: X \to \mathbb R$ is a continuous potential.

Also, it is known a dual variational principle  (Theorem 9.12 in \cite{Wa})
for a continuous self-map $f:X \to X,$ with finite topological entropy. It expresses the measure-theoretic entropy $h_{\mu_0}(f)$  of $f$ through the topological pressure $\Ptop(f,\varphi),$
it says that $h_{\mu_0}(f) =\inf \{\Ptop(f,\varphi)-\int_X \varphi d\mu_0: \varphi \in C(X)\}$
~if~and~only~if~$h_{\mu_0}(f)$~is~upper~semi-continuous~at~$\mu_0,$ where $C(X)$ is the space of all continuous maps $\phi:X \to \mathbb{R}.$

One of the aims of the thermodynamic formalism is to establish the existence of measures with maximal entropy. Such measures, called equilibrium states, include as specific examples Gibbs
 measures and physical measures. With some additional assumptions a classical dynamical system  admits a measure of maximal entropy.

\smallskip

There are several proposals to extend the notions of measure-theoretic entropy and topological pressure to more general settings. One of them are  nonautonomous discrete dynamical
systems (in short: NADDS), determined by  sequences of continuous self-maps $f_{1,\infty}:=\{f_n: X \to X\}_{n \in \mathbb N},$ which are natural generalization of a classical dynamical system. 

Kolyada and Snoha \cite{KS}
introduced and studied the notion of topological entropy of NADDS. Many properties for such dynamical systems
have been studied for many years (cf. \cite{KS},  \cite{HWZ},  \cite{XZ},  \cite{KMS}, \cite{Mou}, \cite{XZ}, \cite{Bis}, \cite{Kaw}).

 For example, Huang, Wen and Zeng \cite{HWZ} introduced and studied the topological pressure of NADDS. Equivalent definitions of topological pressures and basic properties of these notions one can find in \cite{KCL}.

 A main difficulty in setting up the variational principle for NADDS is a lack of common invariant measures, there is no a counterpart of Krylov-Bogoliubov Theorem.

Kawan \cite{Kaw} considers at the same time a sequence of probability spaces $X_{1,\infty}:=\{X_n,\mu_n\}_{n \in \mathbb{N}}$ and a sequence of continuous maps $f_{1,\infty}:=\{f_n: X_n \to X_{n+1}\}_{n \in \mathbb N}$ defined on the sequence of compact metric spaces $\{X_n\}_{n \in \mathbb{N}}.$ A sequence of measures $\mu_{1,\infty}:=\{\mu_n-\text{is a measure on }X_n\}_{n \in \mathbb{N}}$
is called $f_{1,\infty}-$invariant if $f_n\mu_n=\mu_{n+1},$ for any $n \in \mathbb{N}.$ Let $\mathcal{P}_n$ be a finite measurable partition of $X_n$ and let $\mathcal{P}_{1,\infty}:=\{\mathcal{P}_n\}_{n \in \mathbb{N}}$ 
be the sequence of all finite partitions. Recall that the entropy of a finite measurable partition $P = \{P_1,...,P_m\}$ of a probability space $(X,\mathcal{A}, \mu),$ where $\mathcal{A}$ is a
$\sigma-$algebra of $X,$ is defined as $H_{\mu}(P) :=-\sum_{i=1}^{m} \mu(A)\cdot \log(\mu(A)).$ Kawan  defines topological entropy $h_{top}(f_{1,\infty})$ in the same way as Kolyada and Snoha in \cite{KS}. Since  the measure entropy  of a partition $\mathcal{P}_{1,\infty}$ defined as follows
$$h_{\mu_{1}}(f_{1,\infty}, \mathcal{P}_{1,\infty}):=\limsup_{n \to \infty} \frac{1}{n} H_{\mu_1} (\bigvee_{i=0}^{n-1} f_1^{-i}\mathcal{P}_{i+1})$$
can be infinite even for simple systems (see Example 14 and Example 18 in \cite{Kaw}) he chooses  a sufficiently nice subclass $\mathcal{E}_{max}$ from the class of all sequences $\mathcal{P}_{1,\infty},$ called maximal admissible class (for definition see Subsection 3.2 of \cite{Kaw}), to get a reasonable notion of metric entropy. But to get partial variational principle he even chooses a subclass 
$\mathcal{E}_{M} \subset \mathcal{E}_{max},$ called Misiurewicz class and defined in Subsection 4.1 of \cite{Kaw}, to define measure entropy 
$$h_{\mathcal{E}_{M}}(f_{1,\infty}):=\sup_{\mathcal{P}_{1,\infty} \in \mathcal{E}_{M}}h_{\mu_{1}}(f_{1,\infty}, \mathcal{P}_{1,\infty}).$$ His main result, i.e. the partial variational principle, is as follows.\\
Theorem (28 in \cite{Kaw}). For an equicontinuous topological nonautonomous dynamical system $(X_{1,\infty}, f_{1,\infty})$ with $f_{1,\infty}-$invariant sequence $\mu_{1,\infty}$ it holds that
$$h_{\mathcal{E}_{M}}(f_{1,\infty}) \le h_{top}(f_{1,\infty}).$$
\color{black}

 Local measure entropies, in sense of Brin and Katok \cite{BK}, are well defined for  discrete nonautonomous dynamical \color{black} systems. In this context, there were several approaches to relate the topological entropy of a system to its local measure entropies. 

For example, some generalization of the result of Kawan \cite{Kaw} was given by Bi\'s \cite{Bis} who applied local lower measure entropy $h_{\mu}^{Low}(f_{1,\infty})(x)$ and 
local upper measure entropy $h_{\mu}^{Upp}(f_{1,\infty})(x)$ of the system at a point $x \in X$, with respect to a Borel probability measure $\mu$ on $X.$ 

Theorem 1 in \cite{Bis} says that for a Borel probability measure $\mu$ on $X$ a subset $E \subset X$ and $s>0$ one has:
If for any point $x \in E$ the inequality $h_{\mu}^{Low}(f_{1,\infty})(x) \ge s$ holds and $\mu(E)>0,$ then $h_{top}(f_{1,\infty},E) \ge s.$

Feng and Huang \cite{FH} proved that the topological entropy of NADDS coincides with the supremum of integrals of lower local measure entropy function, where the supremum is taking over all Borel probability measures with full support.

\bigskip

In the paper, we prove that for the topological pressure of a  nonautonomous dynamical system the abstract variational principle and dual abstract variational principle (Theorem A) hold. In Theorem B, assuming that the classical  variational principle holds for topological pressure, we show the relations between different measure entropies and invariant measures of the system. Next, we use Misiurewicz construction and introduce Misiurewicz pressure and Misiurewicz entropy. In Theorem C we prove that for Misiurewicz pressure the abstract variational principle and abstract dual variational principle hold as well. In Theorem D, assuming that the classical  variational principle holds for Misiurewicz pressure, we show the relations between different measure entropies and invariant measures of the system.

The paper is organized as follows.
In Section 2, we introduce Convex Analysis approach to NSDDS. In particular we present necessary notions and ideas to formulate Theorem A on abstract variational principle for the topological 
pressure of a nonautonomous dynamical system $f_{1,\infty}:=\{f_n: X \to X\}_{n \in \mathbb N}$. We show that topological pressure determines a pressure function. The variational principle stated in Theorem A for topological pressure ensures that there always exist  measures for which the right-hand side of (2.1) attains the maximum. In Corollary 1 we state that such a measure is unique. We define the space $\mathcal{P}_{f_{1,\infty}}(X)$  of $f_{1,\infty}-$invariant measures and provide
examples (Example 2.3 and Example 2.4)  of nonautonomous dynamical systems for which this space $\mathcal{P}_{f_{1,\infty}}(X)$ is not empty.
In Theorem B, under the assumption that for a nonautonomous dynamical system $f_{1,\infty}:=\{f_n: X \to X\}_{n \in \mathbb N}$ we have also classical variational principle for topological pressure, we present
relations between different entropies of the system. In particular we prove that  $f_{1,\infty}-$invariant measures are determined by the positivity of the abstract measure entropy.
Next we define Misiurewicz entropy and Misiurewicz pressure of a nonautonomous dynamical system $f_{1,\infty}:=\{f_n: X \to X\}_{n \in \mathbb N}.$   We show that Misiurewicz pressure determines a pressure function. We state Theorem C, expressed for the Misiurewicz pressure, which is an analogue of Theorem A expressed for the topological pressure of a nonautonomous dynamical system.
The variational principle stated in Theorem C for Misiurewicz pressure ensures that there always exist  measures for which the right-hand side of (2.6) attains the maximum. In Corollary 2 we state that such a measure is unique. Next we state Theorem D, expressed for the Misiurewicz pressure, which is an analogue of Theorem B expressed for the topological pressure. 

 Section 3 is devoted to the main result of Bi\'s et al. \cite{BCMV}, where 
the authors used methods from Convex Analysis to obtain upper semi-continuous  entropy-like maps adapted to any pressure function on a suitable Banach space, and without summoning any dynamics. An abstract pressure function (Definition 3.2) defined on a Banach space is introduced. In Theorem 3.3 stated for an abstract pressure function we obtain an abstract measure entropy,
abstract variational principle and abstract dual variational principle.

In Section 4, we introduce in details the notion of topological pressure of NADDS, apply the results of  \cite{BCMV}, and provide proofs of  Theorem A and Theorem B.
In Section 5, we introduced   Misiurewicz pressure of a nonautonomous dynamical system, based on \cite{Mis}, and present proofs of Theorem C and Theorem D. \color{black}
\section{Main results}
In this self-contained paper we present two different definitions of  pressures of NADDS determined by continuous self-maps $f_{1,\infty}:=\{f_n: X \to X\}_{n \in \mathbb N}.$ The first one is called the topological pressure (see: Definition 4.1) and  the second one is called the  Misiurewicz \color{black}pressure (see: Definition 5.1). To both pressures 
 we apply the main results of \cite{BCMV}, where the authors used methods from Convex Analysis to obtain upper semi-continuous  entropy-like maps adapted to any abstract  pressure function on a suitable Banach space, and without summoning any dynamics.

For a compact metric space $(X,d)$ we denote by $P(X)$ the space of all Borel probability measures on $X,$ and 
by $C(X)$ the space of all continuous maps $\varphi: X \to \mathbb{R},$ called potentials.
Let $\Ptop(f_{1,\infty}, \,\varphi)$  (resp., $h_{top}(f_{1,\infty})$) denote the topological pressure (resp., the topological entropy) a NADDS $f_{1,\infty}= \{ f_n: X \to X \}^{\infty}_{n=1},$ with respect to
the potential $\varphi \in C(X).$ 
Applying the strategy of \cite{BCMV} to topological pressure of NADDS we get the following  abstract variational principle and abstract dual variational principle.
\bigskip

\noindent {\bf Theorem A.}
For a  nonautonomous discrete dynamical system  $f_{1,\infty}= \{ f_n: X \to X \}^{\infty}_{n=1},$ defined on a compact metric space $(X,d)$ with $h_{top}(f_{1,\infty}) < +\infty$, there exists
an  upper semi-continuous function ${\mathfrak h}_{f_{1,\infty}}: \,\mathcal{P}(X) \, \to \, \mathbb{R}$ such that
\begin{equation}\label{eq:gvp}
\Ptop(f_{1,\infty}, \,\varphi) = \max_{\mu \, \in \, \mathcal{P}(X)}\, \left\{{\mathfrak h}_{f_{1,\infty}}(\mu) + \int_X  \varphi \, d\mu \right\} \quad \quad \forall \, \varphi \in C(X)
\end{equation}
and
\begin{equation}\label{eq:gve}
{\mathfrak h}_{f_{1,\infty}}(\mu) \,\,=\,\, \inf_{\varphi \, \in \, {C(X)}}\,\, \left\{\Ptop(f_{1,\infty}, \,\varphi) - \int\varphi \, d\mu\right\} \quad \quad \forall \,\mu \in  \mathcal{P}(X).
\end{equation}

\bigskip

\noindent {\bf Corollary 1.} For a  nonautonomous  discrete dynamical system $f_{1,\infty}= \{ f_n: X \to X \}^{\infty}_{n=1},$ with finite topological entropy  $h_{top}(f_{1,\infty}),$
and a continuous potential $\varphi \in C(X)$ there exists a unique measure $\mu \in P(X)$ such that $\Ptop(f_{1,\infty}, \,\varphi) =  {\mathfrak h}_{f_{1,\infty}}(\mu) + \int_X  \varphi \, d\mu.$
\color{black}

\bigskip

In the next theorem we introduce a new entropy like quantity $h^*_{f_{1,\infty}}(\mu)$ (for definition see Subsection 4.3) for a NADDS and we present some properties of an abstract measure-entropy of nonautonomous dynamical systems. 

 Let $\mathcal{P}_{f_{1,\infty}}(X)$ be the space of $f_{1,\infty}-$invariant measures, i.e.
 $$\mathcal{P}_{f_{1,\infty}}(X):=\{\mu \, \in \, \mathcal{P}(X): \int (\psi \circ f_n)d\mu=\int \psi d\mu, \forall_{n \in \mathbb{N}},\psi \in C(X) \}.$$

Below we present examples of NADDS admitting invariant measures.
\begin{example} Assume that $(X,d)$ is a compact metric space and there are two  commuting continuous maps $f,g :X \to X,$  that are different from the identity map on $X.$
Then, a Borel probability measure $\mu$ which is $f-$invariant, is $g-$invariant as well (see \cite{PWal} p. 98). So, for $f_{1,\infty}= \{ f_n:X \to X \}^{\infty}_{n=1}$ consisting of pairwise commuting maps (i.e. $f_n \circ f_m=f_m \circ f_n$ for any $m,n \in \mathbb{N}$) there exists  a common $f_{1,\infty}-$invariant measure
$\nu$ and  $\mathcal{P}_{f_{1,\infty}}(X) \not= \emptyset.$

\end{example}
\begin{example}
 The orthogonal group $O(m)$ acting on the $m-$dimensional sphere  $S^m$ is a non-abelian group admitting  $O(m)-$invariant Haar measure. Therefore, for any
$f_{1,\infty}= \{ f_n: S^m \to S^m \}^{\infty}_{n=1},$ with $f_n \in O(m)$ for any $n \in \mathbb{N},$  the Haar measure is $f_{1,\infty}-$ invariant and  $\mathcal{P}_{f_{1,\infty}}(X) \not= \emptyset.$
\end{example}
\color{black}

\bigskip
\noindent {\bf Theorem B.}
\emph{Given a  nonautonomous discrete dynamical system  $f_{1,\infty}= \{ f_n: X \to X \}^{\infty}_{n=1},$ defined on a compact metric space $(X,d).$
 Let $\Ptop(f_{1,\infty}, \cdot): C(X) \to \mathbb{R}$ be the pressure function and ${\mathfrak h}_{f_{1,\infty}}: \mathcal{P}(X) \to \mathbb{R}$  its variational metric entropy given by the equality (2.2).
  Assume that  $\mathcal{P}_{f_{1,\infty}}(X) \not= \emptyset$ \color{black}  and   $\Ptop(f_{1,\infty}, \cdot)$ also satisfies a variational principle
\begin{equation}\label{eq:cvp}
\Ptop(f_{1,\infty}, \,\varphi) = \sup_{\mu \, \in \, \mathcal{P}_{f_{1,\infty}}(X)}\, \left\{h_{f_{1,\infty}}(\mu) + \int \varphi \, d\mu \right\} \quad \quad \forall \, \varphi \in C(X)
\end{equation}
with respect to a non-negative metric entropy map $h_{f_{1,\infty}}: \mathcal{P}_{f_{1,\infty}}(X) \to [0, +\infty[$. Then
\begin{itemize}
\item[(a)] If $\varphi \in C(X)$ and $\mu_\varphi \in \mathcal{P}(X)$ attains the maximum in
(2.3), then $\mu_\varphi$ is $f_{1,\infty}$-invariant, that is,
$$\int (\psi\circ f_n) \,d\mu_\varphi = \int \psi \,d\mu_\varphi \quad \quad \forall\, n \in \mathbb{N} \quad \forall \, \psi \in C(X).$$
\smallskip
\item[(b)] Given $\mu \in \mathcal{P}(X)$, one has $\,\,\mu \in \mathcal{P}_{f_{1,\infty}}(X)$ if and only if ${\mathfrak h}_{f_{1,\infty}}(\mu) \, \geqslant \, 0$.
\smallskip
\item[(c)] For every $\mu \in \mathcal{P}_{f_{1,\infty}}(X)$, one has $\,\,0 \,\leqslant \, h_{f_{1,\infty}}(\mu) \,\leqslant\, {\mathfrak h}_{f_{1,\infty}}(\mu) \leqslant\, {h}^*_{f_{1,\infty}}(\mu) $.
\smallskip
\item[(d)] For every $\varphi \in C(X)$
$$\Ptop(f_{1,\infty}, \varphi) \,=\, \max_{\mu \,\in \,\mathcal{P}_{f_{1,\infty}}(X)}\,\left\{{\mathfrak h}_{f_{1,\infty}}(\mu) + \int \varphi \, d\mu\right\} \, \,= \max_{\mu \,\in \,\mathcal{P}_{f_{1,\infty}}(X)}\,\left\{{h}^*_{f_{1,\infty}}(\mu) + \int \varphi \, d\mu\right\} \,.$$
\end{itemize}}
\bigskip
In the second part of the paper we consider Misiurewicz pressure $P_{Mis}(f_{1,\infty}, \varphi)$ (see  Definition 5.1)
and Misiurewicz entropy $h_{Mis}(f_{1,\infty})$ of $f_{1,\infty}= \{ f_n: X \to X \}^{\infty}_{n=1},$ defined by the equality  $h_{Mis}(f_{1,\infty}):=P_{Mis}(f_{1,\infty}, \phi_0),$ where $\phi_0 \equiv 0.$ We get the following results, similar to Theorem A and Theorem B.
\bigskip
\color{black}

\noindent {\bf Theorem C.}
Given a  nonautonomous discrete dynamical system  $f_{1,\infty}= \{ f_n: X \to X \}^{\infty}_{n=1},$ defined on a compact metric space $(X,d),$ with $h_{Mis}(f_{1,\infty},X) < +\infty.$ Then, there exists
an  upper semi-continuous function ${\mathfrak h}^{Mis}_{f_{1,\infty}}: \,\mathcal{P}(X) \, \to \, \mathbb{R}$ such that
\begin{equation}\label{eq:gvp}
P_{Mis}(f_{1,\infty},\varphi)= \max_{\mu \, \in \, \mathcal{P}(X)}\, \left\{{\mathfrak h}^{Mis}_{f_{1,\infty}}(\mu) + \int_X  \varphi \, d\mu \right\} \quad \quad \forall \, \varphi \in C(X)
\end{equation}
and
\begin{equation}\label{eq:gve}
{\mathfrak h}^{Mis}_{f_{1,\infty}}(\mu) \,\,=\,\, \inf_{\varphi \, \in \, {C(X)}}\,\, \left\{P_{Mis}(f_{1,\infty}, \varphi) - \int\varphi \, d\mu\right\} \quad \quad \forall \,\mu \in  \mathcal{P}(X).
\end{equation}

\noindent {\bf Corollary 2.} For a  nonautonomous discrete dynamical system $f_{1,\infty}= \{ f_n: X \to X \}^{\infty}_{n=1},$ with $h_{Mis}(f_{1,\infty})<\infty$ 
and a continuous potential $\varphi \in C(X)$ there exists a unique measure $\mu_1 \in P(X)$ such that $P_{Mis}(f_{1,\infty}, \,\varphi) =  {\mathfrak h}_{f_{1,\infty}}^{Mis}(\mu_1) + \int_X  \varphi \, d\mu_1.$

\bigskip
\color{black}
In the next theorem we introduce a new entropy-like quantity $h^{*Mis}_{f_{1,\infty}}$ ( see Subsection 5.3) for a NADDS and we present some properties of an abstract measure-entropy of
$f_{1,\infty}= \{ f_n: X \to X \}^{\infty}_{n=1}.$
\bigskip

\noindent {\bf Theorem D.}
\emph{Given a nonautonomous discrete  dynamical system  $f_{1,\infty}= \{ f_n: X \to X \}^{\infty}_{n=1},$ defined on a compact metric space $(X,d).$
 Let $P_{Mis}(f_{1,\infty}, \cdot): C(X) \to \mathbb{R}$ be the Misiurewicz pressure function and ${\mathfrak h}^{Mis}_{f_{1,\infty}}: \mathcal{P}(X) \to \mathbb{R}$ denote its variational metric entropy given by the equality \eqref{eq:gve}. Assume that  $\mathcal{P}_{f_{1,\infty}}(X) \not= \emptyset$ \color{black} and $P_{Mis}(f_{1,\infty}, \cdot)$ also satisfies a variational principle
\begin{equation}\label{eq:cvp}
P_{Mis}(f_{1,\infty}, \varphi) = \sup_{\mu \, \in \, \mathcal{P}_{f_{1,\infty}}(X)}\, \left\{h^{Mis}_{f_{1,\infty}}(\mu) + \int \varphi \, d\mu \right\} \quad \quad \forall \, \varphi \in C(X)
\end{equation}
with respect to a non-negative metric entropy map $h^{Mis}_{f_{1,\infty}}: \mathcal{P}_{f_{1,\infty}}(X) \to [0, +\infty[$. Then
\begin{itemize}
\item[(a)] If $\varphi \in C(X)$ and $\mu_\varphi \in \mathcal{P}(X)$ attains the maximum in \eqref{eq:gvp}, then $\mu_\varphi$ is $f_{1,\infty}$-invariant, that is,
$$\int (\psi\circ f_n) \,d\mu_\varphi = \int \psi \,d\mu_\varphi \quad \quad \forall\, n \in \mathbb{N} \quad \forall \, \psi \in C(X).$$
\smallskip
\item[(b)] Given $\mu \in \mathcal{P}(X)$, one has $\,\,\mu \in \mathcal{P}_{f_{1,\infty}}(X)$ if and only if ${\mathfrak h}^{Mis}_{f_{1,\infty}}(\mu) \, \geqslant \, 0$.
\smallskip
\item[(c)] For every $\mu \in \mathcal{P}_{f_{1,\infty}}(X)$, one has $\,\,0 \,\leqslant \, h^{Mis}_{f_{1,\infty}}(\mu) \,\leqslant\, {\mathfrak h}^{Mis}_{f_{1,\infty}}(\mu) \leqslant\, {h}^{*Mis}_{f_{1,\infty}}(\mu) $.
\smallskip
\item[(d)] For every $\varphi \in C(X)$
$$P_{Mis}(f_{1,\infty}, \varphi) \,=\, \max_{\mu \,\in \,\mathcal{P}_{f_{1,\infty}}(X)}\,\left\{{\mathfrak h}^{Mis}_{f_{1,\infty}}(\mu) + \int \varphi \, d\mu\right\} \, \,= \max_{\mu \,\in \,\mathcal{P}_{f_{1,\infty}}(X)}\,\left\{{h}^{*Mis}_{f_{1,\infty}}(\mu) + \int \varphi \, d\mu\right\} \,.$$
\end{itemize}}
\bigskip

\begin{remark} In context of group actions or semigroup actions results similar to Theorem B or Theorem D one can find in \cite{BCMV2}.
\end{remark}

\section{Abstract variational principle}\label{se:statements}
This section is on the main result of Bi\'s et al. \cite{BCMV}, where
the authors used methods from convex analysis to obtain upper semi-continuous  entropy-like maps adapted to any pressure function on a suitable Banach space, and without summoning any dynamics. Thus, this strategy is unspecific enough to be applied to either the classical topological pressure associated to a single dynamics (as done in \cite{BCMV}) or to the recently defined notions of topological pressure and  Misiurewicz \color{black}pressure for  nonautonomous dynamical systems, we address in this work. Such an abstract variational principle can be applied to obtain variational principles using either the space of probability measures or the space of finitely additive set functions. To present the main result of  \cite{BCMV}, let us first introduce the basis notions and notations.

Let $(X,d)$ be a locally compact metric space (with the distance $d$), $\mathfrak{B}$ be its $\sigma$-algebra of Borel sets, $\mathcal{P}(X)$ denote the space of Borel probability measures on $X$ with the weak$^*$-topology, $\mathcal{P}_a(X)$ stand for the set of Borel real-valued normalized finitely additive set functions with the total variation distance (which we will simply call \emph{finitely additive probabilities}), and $C(X)$ the space of continuous maps $\psi \,\colon X \to \mathbb{R}$ with the supremum norm (whose elements are called potentials).

\smallskip

 In what follows we will consider a Banach space ${\mathbf B}$ over the field $\mathbb{R}$ equal to either
\begin{eqnarray}\label{def:Banach}
L^\infty(X) &=& \big\{\varphi: \, X \,\to\, \mathbb{R} \,\,|\,\, \varphi \text{ is measurable and bounded}\big\} \nonumber\\
&\smallskip& \\
\text{ or } \,\,\,\,C_c(X) &=& \big\{\varphi \in C(X) \,\,|\,\, \varphi \text{ has compact support}\big\} \nonumber
\end{eqnarray}
endowed with the norm $\|\varphi\|_\infty = \sup_{x \, \in \, X}\, |\varphi(x)|$. The Riesz representation theorem asserts that the dual of $C_c(X)$ is identified with the collection of all finite signed measures on $(X, \mathfrak{B})$, whose positive normalized continuous functionals correspond to the space $\mathcal{P}(X)$. It is known that the latter set is compact when equipped with the weak$^*$ topology. On the other hand, the dual of $L^\infty(X)$ is isometrically isomorphic to the space of regular finitely additive bounded signed measures on $(X, \mathfrak{B})$ with the total variation norm, whose subset of positive normalized elements is represented by $\mathcal{P}_a(X)$.

\begin{definition}\label{def:pressure-function}
\emph{A function $\Gamma: {\mathbf B} \to \mathbb{R}$ is called a \emph{pressure function} if it satisfies the following conditions:
\begin{enumerate}
\item[(C$_1$)] \emph{Increasing}: $\,\,\varphi \leqslant \psi \quad \Rightarrow \quad \Gamma(\varphi) \leqslant \Gamma(\psi) \quad \forall\,\varphi, \, \psi \in {\mathbf B}$.
\medskip
\item[(C$_2$)] \emph{Translation invariant}: $\,\,\Gamma(\varphi + c) = \Gamma(\varphi) + c \quad \forall \,\varphi \in {\mathbf B} \quad \forall \,c \in \mathbb{R}$.
\medskip
\item[(C$_3$)] \emph{Convex}: $\Gamma(t \,\varphi + (1-t) \,\psi) \leqslant t\, \Gamma(\varphi) + (1-t)\, \Gamma(\psi) \quad \forall \,\varphi, \, \psi \in {\mathbf B} \quad \forall\, t \in \,[0,1]$.
\end{enumerate}}
\end{definition}

For a pressure function $\Gamma: {\mathbf B} \to \mathbb{R}$ we get the following.
$$\Gamma(\psi)-\|\varphi-\psi\|_\infty = \Gamma(\psi-\|\varphi-\psi\|_\infty) \leqslant \Gamma(\varphi)\leqslant \Gamma(\psi+\|\varphi-\psi\|_\infty)= \Gamma(\psi)+\|\varphi-\psi\|_\infty,$$
$| \Gamma(\varphi)  - \Gamma(\psi) | \leqslant \|\varphi-\psi\|_\infty$ for every $\varphi,\, \psi\in {\mathbf B}$. The following is the main abstract variational principle and the upper semi-continuous metric entropy map arising from convex analysis that we will use later.

\begin{theorem}\label{thm:main-lemma} \cite[Theorem 1]{BCMV}
Let $(X,d)$ be a metric space and $\Gamma: {\mathbf B} \to \mathbb{R}$ be a pressure function.
\begin{itemize}
\item[(a)] If $X$ is locally compact and ${\mathbf B} = C_c(X)$ then there exists
an  upper semi-continuous entropy map ${\mathfrak h}$ such that
\begin{equation}\label{eq:var-P}
\Gamma(\varphi) \,\,=\,\, \max_{\mu \, \in \, \mathcal{P}(X)}\,\,\, \left\{{\mathfrak h}(\mu) + \int \varphi \, d\mu \right\} \quad \quad \forall\, \varphi \in {C_c(X)}
\end{equation}
Besides,
$${\mathfrak h}(\mu) \,\,=\,\, \inf_{\varphi \, \in \, {C_c(X)}}\,\, \left\{\Gamma(\varphi) - \int\varphi \, d\mu\right\} \quad \quad \forall\, \mu \in  \mathcal{P}(X)$$
and if $\alpha: \mathcal{P}_a(X) \,\to\, \mathbb{R}\cup\{+\infty, -\infty\}$ is another function for which \eqref{eq:var-P} holds then $\alpha\leqslant \mathfrak h$.
Furthermore, if $\cE_\varphi(\Gamma)$ denotes the set of maximizing elements in \eqref{eq:var-P} then there is a residual subset
$\mathfrak R \subset C_c(X)$ such that $\#\,\cE_\varphi(\Gamma) = 1$ for every $\varphi\in \mathfrak R$.
\medskip
\item[(b)] If ${\mathbf B=L^\infty(X)}$, then
\begin{equation}\label{eq:var}
\Gamma(\varphi) \,\,=\,\, \max_{\mu \, \in \, \mathcal{P}_a(X)}\,\,\, \left\{{\mathfrak h}(\mu) + \int \varphi \, d\mu \right\} \quad \quad \forall\, \varphi \in L^\infty(X)
\end{equation}
where
${\mathfrak h}(\mu) \,\,=\,\, \inf_{\varphi \, \in \, {\mathbf B}}\,\, \left\{\Gamma(\varphi) - \int\varphi \, d\mu\right\}$ for every $\mu \in  \mathcal{P}_a(X)$.
Besides, the map ${\mathfrak h}$ is  upper semi-continuous
and maximal among any other upper semi-continuous functions attaining the supremum in ~\eqref{eq:var}.
\end{itemize}
\end{theorem}


\section{Topological entropy and topological pressure of NADDS}

\subsection{Basic definitions and properties} 

In this section we introduce a topological entropy and topological pressure of NADDS
given by a  sequence $f_{1,\infty}= \{ f_n:X \to X \}^{\infty}_{n=1}$ of continuous maps of a compact metric space $(X,d).$ The sequence of self-maps determines a sequence of metrics
$$d_n (x, y) := max\{d( f _i \circ f_{i-1} \circ ...\circ f_1(x), f _i \circ f_{i-1} \circ ...\circ f_1(y)) : 1\le i\le n-1\},$$
where $x, y \in X.$
Following Kolyada and Snoha [13], we define a topological entropy of the NADDS determined by $f_{1,\infty}$
 as follows.  Fix a positive integer $n$ and 
 $\epsilon>0.$ We say that a subset $A \subset X$ is $(n,\epsilon)-$separated if for any distinct points
 $a_1,a_2 \in A$
 $$d_n (a_1, a_2) := max\{d( f _i \circ f_{i-1} \circ ...\circ f_1(a_1), f _i \circ f_{i-1} \circ ...\circ f_1(a_2) : 1\le i\le n-1\}\ge \epsilon.$$
 Let  $s(n, \epsilon, f_{1,\infty}):= max\{card(A): A~is~(n,\epsilon)-separated~subset~of~X\}.$
 
 \begin{definition} The quantity
 $$h_{top}(f_{1,\infty}):=\lim_{\epsilon \to 0} \limsup_{n \to \infty} \frac{1}{n} log (s(n, \epsilon, f_{1,\infty}))$$
 is called the {\it topological entropy} of a NADDS determined by $f_{1,\infty}.$
 
 \end{definition}

{\it Topological pressure} of a NADDS, determined by $f_{1,\infty},$ with respect to a potential $\varphi \in C(X)$ we define as follows
$$ P_{top}(f_{1,\infty}, \varphi):=\lim_{\varepsilon \to 0^{+}} \limsup_{n \to \infty} \frac{1}{n} \log P_n(f_{1,\infty}, \varphi ,\varepsilon),$$ 
where for any $n\in \mathbb N$ and $\varepsilon >0$
$$P_n(f_{1,\infty}, \varphi ,\varepsilon):=\sup _{E}\big \{\sum_{x \in E}e^{S_n^{f_{1,\infty}}\varphi(x)}: E~\text{is}~(n, \varepsilon)-\text{separated}~\text{subset~of}~X\}$$
and
$$S_n^{f_{1,\infty}}\varphi(x):=\varphi(x)+\varphi(f_1(x))+\varphi(f_2\circ f_1(x))+\cdots +\varphi(f_{n-1}\circ f_{n-2} \circ...\circ f_1(x)).$$
 Notice that for the potential $\varphi_0 \equiv 0$ one has
 $$ P_{top}(f_{1,\infty}, \varphi_0)=h_{top}(f_{1,\infty}).$$
  From the definition of pressure we get for any $\varphi \in C(X)$ the following inequalities
 $$h_{top}(f_{1,\infty})+\inf_{x\in X}\varphi(x) \le P_{top}(f_{1,\infty}, \varphi) \le h_{top}(f_{1,\infty})+\sup_{x\in X}\varphi(x).$$ \color{black}
The topological pressure has the following properties.
\begin{lemma} $If~\varphi \leq\psi \Rightarrow P_{top}(f_{1,\infty}, \varphi)\leq  P_{top}(f_{1,\infty}, \psi),$ for any $\varphi, \psi \in C(X).$
\end{lemma}
\begin{proof} Let $n\in \mathbb N$ and $x\in X$. If $\varphi \leq\psi $, then 
\begin{align*}
S_n^{f_{1,\infty}}\varphi(x)&=  \varphi(x)+\varphi(f_1(x))+\varphi(f_2\circ f_1(x))+\cdots +\varphi(f_{n-1}\circ f_{n-2} \circ...\circ f_1(x)).        \\
&\leq \psi(x)+\psi(f_1(x))+\psi(f_2\circ f_1(x))+\cdots +\psi(f_{n-1}\circ f_{n-2} \circ...\circ f_1(x)). \\
&=S_n^{f_{1,\infty}}\psi(x),
\end{align*}
therefore
\begin{align*}
P_n(f_{1,\infty}, \varphi ,\varepsilon)&=\sup _{E}\big \{\sum_{x \in E}e^{S_n^{f_{1,\infty}}\varphi(x)}: E~\text{is}~(n, \varepsilon)-\text{separated}~\text{subset~of}~X\}\\
&\leq \sup _{E}\big \{\sum_{x \in E}e^{S_n^{f_{1,\infty}}\psi(x)}: E~\text{is}~(n, \varepsilon)-\text{separated}~\text{subset~of}~X\}\\
&=P_n(f_{\infty}, \psi ,\varepsilon).
\end{align*}
Passing to the  limsup with $n\rightarrow \infty$ and limit with $\varepsilon \to 0^+$ we get
$$ P_{top}(f_{1,\infty}, \varphi)\leq  P_{top}(f_{1,\infty}, \psi).$$

\end{proof}

\begin{lemma} $ P_{top}(f_{1,\infty}, \varphi +c)= P_{top}(f_{1,\infty}, \varphi)+c,$ for any $\phi \in C(X)$ and $c \in \mathbb R.$
\end{lemma}

\begin{proof} Notice that
\begin{align*}
S_n^{f_{1,\infty}}(\varphi +c)(x)&=\varphi(x)+\varphi(f_1(x))+\varphi(f_2\circ f_1(x))+\cdots +\varphi(f_{n-1}\circ f_{n-2} \circ...\circ f_1(x))+nc\\
&=S_n^{f_{1,\infty}}\varphi(x)+nc.
\end{align*}
Therefore,
\begin{align*}
P_n(f_{1,\infty}, \varphi+c ,\varepsilon)&=\sup _{E}\big \{\sum_{x \in E}e^{S_n^{f_{1,\infty}}(\varphi+c)(x)}: E~\text{is}~(n, \varepsilon)-\text{separated}~\text{subset~of}~X\}\\
&=e^{nc} \sup _{E}\big \{\sum_{x \in E}e^{S_n^{f_{1,\infty}}\varphi(x)}: E~\text{is}~(n, \varepsilon)-\text{separated}~\text{subset~of}~X\}\\
&=e^{nc}P_n(f_{1,\infty}, \varphi ,\varepsilon).
\end{align*}
Taking logarithms of both sides and passing to the limsup with $n\rightarrow \infty$ and limit with  $\varepsilon \to 0^+$ we get
$$ P_{top}(f_{1,\infty}, \varphi +c)=  P_{top}(f_{1,\infty}, \varphi)+c.$$
\end{proof}

\begin{lemma} $ P_{top}(f_{\infty} ,t\varphi+(1-t)\psi)\leq t P_{top}(f_{\infty},\varphi)+(1-t) P_{top}(f_{\infty},\psi),~~\forall t \in [0,1]$ and for any $\varphi,\psi \in C(X).$
\end{lemma}

\begin{proof} Choose $t \in [0,1],$ then
\begin{align*}
S_n^{f_{1,\infty}}(t\varphi +(1-t)\psi)(x)&=(t\varphi +(1-t)\psi)(x)+\cdots +(t\varphi +(1-t)\psi)(f_{n-1}\circ f_{n-2} \circ...\circ f_1(x))\\
&=tS_n^{f_{1,\infty}}\varphi(x)+(1-t)S_n^{f_{1,\infty}}\psi(x).
\end{align*}
By the Holder's inequality (with $p=\frac{1}{t}$ and $q=\frac{1}{1-t}$) we get
$$\sum_{x \in E}e^{S_n^{f_{1,\infty}}(t\varphi+(1-t)\psi)(x)}\le (\sum_{x \in E}e^{S_n^{f_{1,\infty}}(\varphi(x))})^{t} \cdot (\sum_{x \in E}e^{S_n^{f_{1,\infty}}(\psi(x))})^{1-t},$$
where $E$ is an $(n,\varepsilon)-$ separated subset of $X.$ Therefore

$$P_n(f_{1,\infty}, t\varphi+(1-t)\psi ,\varepsilon)\\
\le P_n(f_{\infty}, \varphi ,\varepsilon)^t \cdot P_n(f_{\infty}, \psi ,\varepsilon)^{1-t}.$$
Taking logarithms of both sides and passing to limsup with $n\rightarrow \infty$ and limit with $\varepsilon \to 0^+$ we obtain
$$P_{top}(f_{1,\infty} ,t\varphi+(1-t)\psi)\leq t P_{top}(f_{1,\infty},\varphi)+(1-t) P_{top}(f_{1,\infty},\psi).$$
\end{proof}
As a  consequence  \color{black} we get the following    property.  \color{black}
\begin{corollary} Given a NADDS determined by  $f_{1,\infty}= \{ f_n:X \to X \}^{\infty}_{n=1}$  with \\ $h_{top}(f_{1,\infty}) < +\infty,$ then $\Ptop(f_{1,\infty}, \cdot)$ is a pressure function.
\end{corollary}

\bigskip
\subsection{Proof of Theorem A} 
\begin{proof}Since $ P_{top}(f_{1,\infty}, \cdot) : C(X) \to \mathbb{R}$ is a pressure function
the unifying statement of Theorem~\ref{thm:main-lemma} provides an abstract variational principle 
for a nonautonomous discrete dynamical system determined by $f_{1,\infty}= \{ f_n:X \to X \}^{\infty}_{n=1}$  and the topological pressure $P_{top}(f_{1,\infty}, \varphi),$
with respect to $\varphi \in C(X).$
\end{proof}

\subsection{Proof of Corollary 1}

For a pressure function $\Gamma: C(X) \to \mathbb{R}$ and potential $\phi \in C(X)$ we say  that $\mu \in P(X)$ is a {\it tangent functional }to $\Gamma$ at $\phi$ if for any $\psi \in C(X)$ the inequality
$\Gamma(\phi + \psi)-\Gamma(\phi) \ge \int\psi d\mu$ holds. Let $\mathcal{T}_{\phi}(\Gamma)$ be the space of tangent functionals to $\Gamma$ at $\phi$ and let
 $$\mathcal{E}_{\phi}(\Gamma):=\{\mu \in P(X): \Gamma(\phi)={\mathfrak h}(\mu) +\int \psi d\mu\}.$$
Let us recall   the  definition of the Gateaux differential (or Gateaux derivative), which  is a generalization of the concept of directional derivative in differential calculus.
Let $F$ be a function on an open set $U$ of a Banach space $X$ into the Banach space $Y,$ {\it the Gateaux differential} $dF(u;\psi )$ of $F$  at $u \in U$  in the direction $\psi \in X$ is defined as
$$dF(u;\psi )=\lim_{t\to 0} \frac{F(u +t\cdot \psi)-F(u)}{t}.$$
If the limit exists for all $\psi \in X,$ then one says that $F$ is Gateaux differentiable at $u.$
\bigskip

Now we can provide the proof of Corollary 1.

\begin{proof} Fix a  nonautonomous  discrete dynamical system $f_{1,\infty}= \{ f_n: X \to X \}^{\infty}_{n=1},$ with finite topological entropy  $h_{top}(f_{1,\infty}),$ and a continuous potential $\varphi \in C(X).$
From Theorem A it follows that for the pressure function $P_{top}(f_{1,\infty}, \cdot): C(X) \to \mathbb{R}$ and a continuous potential $\varphi \in C(X)$ there exists a  measure $\mu \in P(X)$ such that $\Ptop(f_{1,\infty}, \,\varphi) =  {\mathfrak h}_{f_{1,\infty}}(\mu) + \int_X  \varphi \, d\mu.$ It remains to show that this measure $\mu$ is uniquely determined.
Since the function $ P_{top}(f_{1,\infty}, \cdot): C(X) \to \mathbb{R}$ is convex, for any $\phi, \psi \in C(X)$ we obtain that the function
$$ t  \to \frac{P_{top}(f_{1,\infty}, \phi+t\cdot  \psi)-P_{top}(f_{1,\infty}, \phi)}{t}$$
is increasing and therefore there exist limits:
$$d^+P_{top}(f_{1,\infty}, \phi)(\psi):=\lim_{t \searrow 0^+} \frac{P_{top}(f_{1,\infty}, \phi+t\cdot  \psi)-P_{top}(f_{1,\infty}, \phi)}{t},$$
$$d^-P_{top}(f_{1,\infty}, \phi)(\psi):=\lim_{t \nearrow 0^-} \frac{P_{top}(f_{1,\infty}, \phi+t\cdot  \psi)-P_{top}(f_{1,\infty}, \phi)}{t},$$
hence the pressure function $P_{top}(f_{1,\infty}, \cdot)$ is Gateaux differentialble at $\phi$.

From Theorem 2 in \cite{BCMV} it follows that  for the pressure function $\Gamma(\cdot)=P_{top}(f_{1,\infty}, \cdot)$ we get equality $\mathcal{E}_{\phi}(\Gamma) =\mathcal{T}_{\phi}(\Gamma),$ for any 
$\phi \in C(X).$ Using the Gateaux differentiability of the pressure function $P_{top}(f_{1,\infty}, \cdot)$ and the Corollary 4 in \cite{BCMV} we get that for $\phi$ there exists a 
unique tangent functional in $\mathcal{T}_{\phi},$ that completes the proof.
\end{proof}

\bigskip

Analyzing the proof of Corollary 1,  it can be seen that for a  nonautonomous  discrete dynamical system $f_{1,\infty}= \{ f_n: X \to X \}^{\infty}_{n=1},$ defined on a compact metric space $(X,d),$ 
an abstract pressure function $\Gamma_{f_{1,\infty}}: C(X) \to \mathbb{R}$ 
with $\Gamma_{f_{1,\infty}}(0)< \infty$ and a potential $\varphi \in C(X)$ there exists a unique measure $\mu$ on $X$
such that $\Gamma_{f_{1,\infty}}( \varphi) =  {\mathfrak h}_{f_{1,\infty}}(\mu) + \int_X  \varphi \, d\mu$ if and only if $\Gamma_{f_{1,\infty}}: C(X) \to \mathbb{R}$ is Gateaux differentiable at $\varphi.$

\bigskip

In general case for a nonautonomous dynamical system $f_{1,\infty}= \{ f_n:X \to X \}^{\infty}_{n=1}$, there is no common  $f_n-$invariant Borel probability measure, for any 
$n \in \mathbb N$ (see Example 4.6).

\begin{example} 
Consider $f_{1,\infty}= \{ f_n:S^1\to S^1 \}^{\infty}_{n=1}.$ Assume  that $f_i,$ where $i=1,2,$ are circle diffeomorphisms, such that each of them has two fixed points:
a source $p_i$ and a sink $q_i$ which split $S^1$ into two arcs $A_i$ and $B_i$ 
each of them being contracted to $q_i.$ Additionally, we assume that $\{p_1,q_1\} \cap \{p_2,q_2\}=\emptyset.$
By Proposition 3.1.6 in \cite{PWal} any measure $\mu$ that is $f_i-$invariant
has to be supported in the set $\{p_i,q_i\}.$ Therefore $f_1$ and $f_2$ do not have a common invariant measure, which consequently gives that $\mathcal{P}_{f_{1,\infty}}(X)= \emptyset.$

\end{example}

However, there are many special cases of nonautonomous dynamical systems for which invariant measures exist (see Example 2.3 and Example 2.4). \color{black}

\bigskip

\subsection{Proof of Theorem B}
Theorem B generalizes the analogous property established in \cite{BCMV} for a single dynamics.
Theorem B result relates the restriction to $\mathcal{P}_{f_{1,\infty}}(X)$ of the variational metric entropy ${\mathfrak h}_{f_{1,\infty}}$ with the star-entropy of NADDS, defined by
$$\mu\,\in\, \mathcal P_{f_{1,\infty}}(X) \quad \mapsto \quad h^*_{f_{1,\infty}}(\mu) \,:=\, \sup \Big\{\limsup_{n \, \to  \,+\infty}\,h_{f_{1,\infty}}(\mu_n) \,\,|\,\,\mu_n \in \mathcal{P}_{f_{1,\infty}}(X) \text{ and } \lim_{n \, \to \, +\infty}  \mu_n = \mu\Big\},$$
where the convergence of $(\mu_n)_{n \, \in \mathbb{N}}$ takes place in the weak$^*$-topology. 
\bigskip
Now we can present a proof of Theorem B.
\begin{proof}
Let  $f_{1,\infty}= \{ f_n: X \to X \}^{\infty}_{n=1}$ be a nonautonomous discrete dynamical system, defined on a compact metric space $X,$ and let  $\Ptop(f_{1,\infty}, \,\cdot)$ be the topological pressure function.

\noindent (a) As $\Ptop(f_{1,\infty}, \,\cdot)$ satisfies the variational principle (2.5)
 with respect to a common $f_{1,\infty}$-invariant Borel probability measures $\mu$ on $X$,
it is easy to conclude that, for every $\varphi, \, \psi \in C(X)$ and every $f_n$, we have
\begin{equation}\label{eq:property}
\Ptop(f_{1,\infty}, \,\varphi + \psi \circ f_n- \psi) \,=\, \Ptop(f_{1,\infty},\, \varphi) \,=\,\Ptop(f_{1,\infty}, \varphi + \psi - \psi \circ f_n).
\end{equation}
Indeed, by the variational principle we get
$$\Ptop(f_{1,\infty}, \,\varphi + \psi \circ f_n - \psi) = \sup_{\mu \, \in \, \mathcal{P}_{f_{1,\infty}}(X)} \left\{h_{f_{1,\infty}}(\mu) + \int (\varphi + \psi \circ f_n - \psi) \, d\mu \right\} =$$
$$ \sup_{\mu \, \in \, \mathcal{P}_{f_{1,\infty}}(X)} \left\{h_{f_{1,\infty}}(\mu) + \int \varphi \, d\mu \right\}=\Ptop(f_{1,\infty},\, \varphi),$$
where the second equality is due to the assumption that the measure $\mu$ is $f_n-$invariant
and similarly with regard to the second equality in \eqref{eq:property}.

\smallskip

Now we claim that every probability measure attaining the maximum in 
(2.1) is $f_{1,\infty}$-invariant. Consider $\psi \in C(X)$, $f_n:X \to X$ and fix $\mu_\varphi, \mu_1, \, \mu_2 \in \mathcal{P}(X)$ provided by
(2.1) such that
\begin{eqnarray*}
\Ptop(f_{1,\infty},\varphi) &=& {\mathfrak h}_{f_{1,\infty}}(\mu_\varphi) + \int \varphi \, d\mu_\varphi\\
\Ptop(f_{1,\infty},\varphi + \psi \circ f_n - \psi) &=& {\mathfrak h}_{f_{1,\infty}}(\mu_1) + \int \varphi\, d\mu_1 + \int (\psi\circ f_n) \, d\mu_1 - \int \psi \, d\mu_1 \\
\Ptop(f_{1,\infty},\varphi + \psi - \psi \circ f_n) &=& {\mathfrak h}_{f_{1,\infty}}(\mu_2) + \int \varphi \, d\mu_2 + \int \psi \, d\mu_2 - \int (\psi\circ f_n) \, d\mu_2.
\end{eqnarray*}
The first two equalities, property \eqref{eq:property} and the variational principle
(2.1) now yield
\begin{eqnarray*}
{\mathfrak h}_{f_{1,\infty}}(\mu_\varphi) + \int \varphi \, d\mu_\varphi &=& {\mathfrak h}_{f_{1,\infty}}(\mu_1) + \int \varphi \, d\mu_1 + \int (\psi\circ f_n) \, d\mu_1 - \int \psi \, d\mu_1  \\
&\geqslant&  {\mathfrak h}_{f_{1,\infty}}(\mu_\varphi) + \int \varphi \, d\mu_\varphi + \int (\psi \circ f_n) \, d\mu_\varphi - \int \psi \, d\mu_\varphi
\end{eqnarray*}
and so $\int (\psi \circ f_n) \, d\mu_\varphi - \int \psi \, d\mu_\varphi \,\leqslant\, 0.$ In a similar way, we deduce that
\begin{eqnarray*}
{\mathfrak h}_{f_{1,\infty}}(\mu_\varphi) + \int \varphi \, d\mu_\varphi &=& {\mathfrak h}_{f_{1,\infty}}(\mu_2) + \int \varphi \, d\mu_2 + \int \psi \, d\mu_2  - \int (\psi\circ f_n) \, d\mu_2 \\
&\geqslant&  {\mathfrak h}_{f_{1,\infty}}(\mu_\varphi) + \int \varphi \, d\mu_\varphi + \int \psi \, d\mu_\varphi - \int (\psi \circ f_n) \, d\mu_\varphi
\end{eqnarray*}
so $\int \psi \, d\mu_\varphi - \int (\psi \circ f_n) \, d\mu_\varphi \,\leqslant\, 0.$ Therefore, for any $n \in \mathbb N$ the equality $\int \psi \, d\mu_\varphi=\int (\psi \circ f_n) \, d\mu_\varphi$ holds, 
it means that the measure
$\mu_\varphi$ is $f_n-$invariant, for any $n \in \mathbb N.$ Thus, $\mu_\varphi$ is $f_{1,\infty}$-invariant.

\medskip

\noindent (b) The proof of this assertion is  similar  to the ones done in \cite[page 222]{Wa} for a single dynamics. Indeed, from the assumptions that $h_{\mathrm{top}}(f_{1,\infty}) < +\infty$ and that the topological pressure satisfies a variational principle with the non-negative metric entropy map $h_{f_{1,\infty}}$ we conclude the following. 
 We claim that for given $\mu \in \mathcal{P}(X)$,
$$\mu \in \mathcal{P}_{f_{1,\infty}}(X) \quad \Leftrightarrow \quad \int \varphi \, d\mu \,\leqslant\, \Ptop(f_{1,\infty},\varphi) \quad \forall\, \varphi \in C(X).$$
Indeed, for $\mu \in \mathcal{P}_{f_{1,\infty}}(X)$ the variational principle yields that the inequality $ \int \varphi \, d\mu \,\leqslant\, \Ptop(f_{1,\infty},\varphi)$ holds for any $\varphi \in C(X).$

Now suppose that $\mu$ is a finite signed measure such that for any $\varphi \in C(X)$  we have $\int \varphi \, d\mu \,\leqslant\, \Ptop(f_{1,\infty},\varphi).$
In the first step we show that $\mu$ takes only non-negative values. To this aim suppose that $\varphi \ge 0, $~$\epsilon>0$ and $n \in \mathbb{N}.$ Then
$$\int n\cdot(\varphi+\epsilon)d \mu=- \int -n\cdot(\varphi+\epsilon)d \mu \ge -\Ptop(f_{1,\infty},-n\cdot(\varphi+\epsilon)) \ge -[h_{top}(f_{1,\infty})  +\sup_{x \in X}(-n\cdot(\varphi+\epsilon) )]=$$
$$-h_{top}(f_{1,\infty}) +n \cdot \inf_{x \in X}(\varphi +\epsilon)>0\quad \text{for~large~n}.$$
Therefore $\int (\varphi +\epsilon) d\mu >0.$ Since $\epsilon$ is arbitrary small we get $\int \varphi  d\mu \ge 0$ and $\mu$ is a measure.

In the next step we show that $\mu$ is a probability measure. Take $n \in \mathbb{Z},$ then 
$$\int n d \mu \le  \Ptop(f_{1,\infty},n) =h_{top}(f_{1,\infty}) +n.$$
So for  $n \in \mathbb{N}$ we get $\mu(X)=\int 1 d\mu\le 1 + \frac{h_{top}(f_{1,\infty})}{n}$ and hence $\mu(X)\le 1.$ On the other hand, for negative $n$ we get
$\mu(X)=\int 1 d\mu\ge 1 + \frac{h_{top}(f_{1,\infty})}{n}$ which yields $\mu(X)\ge 1.$ Therefore we have proved that $\mu(X)= 1.$

In the next step we show that $\mu \in \mathcal{P}_{f_{1,\infty}}(X).$ For any $m \in \mathbb{Z,}$~ arbitrary $\varphi \in C(X)$ and fixed $n \in \mathbb{N}$ due to (4.7) we get
$$m\int (\varphi \circ f_n-\varphi)d \mu \le \Ptop(f_{1,\infty},m\cdot(\varphi \circ f_n-\varphi)) =\Ptop(f_{1,\infty}, 0)=h_{top}(f_{1,\infty}). $$
If $m>1$ then dividing by $m$ and letting $m$ go to $\infty$ we get $\int (\varphi \circ f_n-\varphi)d \mu\le 0.$ If $m<-1$ then dividing by $m$ and letting $m$ go to $-\infty$ we get $\int (\varphi \circ f_n-\varphi)d \mu\ge 0.$ Therefore for any potential $\varphi \in C(X)$ and arbitrary $n \in \mathbb{N}$ we have $\int \varphi \circ f_n d \mu = \int \varphi d\mu,$ so $\mu \in \mathcal{P}_{f_{1,\infty}}(X).$
The claim has been proved.

\color{black}
Therefore, by formula 
(2.2) we get
$$\mu \in \mathcal{P}_{f_{1,\infty}}(X) \quad \Leftrightarrow \quad {\mathfrak h}_{f_{1,\infty}}(\mu) \,\geqslant\, 0.$$
Thus we have proved that the values of the map ${\mathfrak h}_{f_{1,\infty}}$ determine the elements of the set $\mathcal{P}_{f_{1,\infty}}(X)$.
\medskip

\noindent (c) By definition, for every $\mu \in \mathcal{P}_{f_{1,\infty}}(X)$ one has $0 \le h_{f_{1,\infty}}(\mu) \leqslant h^*_{f_{1,\infty}}(\mu)$. Moreover, by the variational principle (2.5),
 for each $\mu \in \mathcal{P}_{f_{1,\infty}}(X)$
$$h_{f_{1,\infty}}(\mu) \,\,\leqslant \,\,\Ptop(f_{1,\infty}, \varphi) - \int \varphi \, d\mu  \quad \quad \forall \,\varphi \, \in \, C(X)$$
and therefore
\begin{equation}\label{ineq}
{\mathfrak h}_{f_{1,\infty}}(\mu) = \inf_{\varphi \, \in \, C(X)}\, \left\{\Ptop(f_{1,\infty}, \varphi) - \int \varphi \, d\mu\right\} \,\,\geqslant \,\, h_{f_{1,\infty}}(\mu).
\end{equation}
Consequently, if $h^*_{f_{1,\infty}}(\mu)  >  {\mathfrak h}_{f_{1,\infty}}(\mu)$ for some $\mu \in \mathcal{P}_{f_{1,\infty}}(X)$, then there would exist $\nu \in \mathcal{P}_{f_{1,\infty}}(X)$ satisfying $h_{f_{1,\infty}}(\nu) > {\mathfrak h}_{f_{1,\infty}}(\nu)$, contradicting ~\eqref{ineq}. This proves that
\begin{equation}\label{eq:inequalities}
0\leqslant h_{f_{1,\infty}}(\mu) \,\leqslant\, h^*_{f_{1,\infty}}(\mu) \,\leqslant\,  {\mathfrak h}_{f_{1,\infty}}(\mu) \quad \quad \forall\, \mu \in \mathcal{P}_{f_{1,\infty}}(X).
\end{equation}

\medskip

\noindent (d) Observe now that the inequalities \eqref{eq:inequalities} imply that
\begin{align*}
\sup_{\mu \,\in \,\mathcal{P}_{f_{1,\infty}}(X)}\,\left\{h_{f_{1,\infty}}(\mu) + \int \varphi \, d\mu\right\}\,\,& \leqslant \,\,\max_{\mu \,\in \,\mathcal{P}_{f_{1,\infty}}(X)}\,\left\{h^*_{f_{1,\infty}}(\mu) + \int \varphi \, d\mu\right\} \\
& \leqslant \,\, \max_{\mu \,\in \,\mathcal{P}_{f_{1,\infty}}(X)}\,\left\{{\mathfrak h}_{f_{1,\infty}}(\mu) + \int \varphi \, d\mu\right\}.
\end{align*}
On the other hand, as $\mathcal{P}_{f_{1,\infty}}(X) \, \subset \, \mathcal{P}(X)$ and ${\mathfrak h}_{f_{1,\infty}}$ is upper semi-continuous, for every $\varphi  \in  C(X)$ we have
$$\max_{\mu \,\in \,\mathcal{P}_{f_{1,\infty}}(X)}\,\left\{{\mathfrak h}_{f_{1,\infty}}(\mu) + \int \varphi \, d\mu \right\} \,\,\leqslant \,\, \max_{\mu \,\in \,\mathcal{P}(X)}\,\left\{{\mathfrak h}_{f_{1,\infty}}(\mu) + \int \varphi \, d\mu\right\}.$$
Therefore, from the variational principles (2.1)
and (2.5)
we obtain
\begin{eqnarray*}
\max_{\mu \,\in \,\mathcal{P}_{f_{1,\infty}}(X)}\,\left\{{\mathfrak h}_{f_{1,\infty}}(\mu) + \int \varphi \, d\mu\right\} \, &\leqslant& \, \max_{\mu \,\in \,\mathcal{P}(X)}\,\left\{{\mathfrak h}_{f_{1,\infty}}(\mu) + \int \varphi \, d\mu\right\} \\
\,&=&\, \Ptop(f_{1,\infty}, \varphi) \\
&=& \sup_{\mu \,\in \,\mathcal{P}_{f_{1,\infty}}(X)}\,\left\{h_{f_{1,\infty}}(\mu) + \int \varphi \, d\mu\right\}\\
&\leqslant&  \,\max_{\mu \,\in \,\mathcal{P}_{f_{1,\infty}}(X)}\,\left\{h^*_{f_{1,\infty}}(\mu) + \int \varphi \, d\mu\right\} \\
&\leqslant& \, \max_{\mu \,\in \,\mathcal{P}_{f_{1,\infty}}(X)}\,\left\{{\mathfrak h}_{f_{1,\infty}}(\mu) + \int \varphi \, d\mu\right\}.
\end{eqnarray*}
Consequently,
$$\Ptop(f_{1,\infty}, \varphi) \,=\, \max_{\mu \,\in \,\mathcal{P}_{f_{1,\infty}}(X)}\,\left\{{\mathfrak h}_{f_{1,\infty}}(\mu) + \int \varphi \, d\mu\right\} \,=\, \max_{\mu \,\in \,\mathcal{P}_{f_{1,\infty}}(X)}\,\left\{h^*_{f_{1,\infty}}(\mu) + \int \varphi \, d\mu\right\}.$$
The proof is complete.
\end{proof}


\section{Misiurewicz entropy and Misiurewicz pressure  for NADDS}
\subsection{Basic definitions} 

In 1976, Misiurewicz  \cite{Mis} defined a topological pressure and a measure-theoretic entropy  for a $\ZZ^n_{+}$ action on a compact space.
His celebrated result says that the variational principle holds for a $\ZZ^n_{+}$ action on a compact space. Our definition of Misiurewicz pressure of $f_{1,\infty}$ was 
inspired by his definition  of the topological pressure   for a $\ZZ^n_{+}$ action \cite{Mis}.

In this section we consider a NADDS $f_{1,\infty}:=\{f_n: X \to X\}_{n \in \mathbb N}$ defined on a compact space $X.$ Let $\mathcal{W}$ be the set of all neighborhoods
of the diagonal $\{(x,x):x \in X\} \subset X \times X$ directed by the inclusion. It is known ( \cite{Ke}) that this is uniform structure for $X.$ 

For any $n \in \mathbb{N},$ neighborhood $\delta \in \mathcal{W},$ potential $\phi \in C(X)$ and any finite subset $E \subset X$ we define:
$$\delta_n:=\bigcap_{i=1}^n( f_i\circ...\circ f_1  \times  f_i\circ...\circ f_1)^{-1} \delta,$$
$$\phi_n(x):=\sum_{i=1}^n \phi \circ  f_i\circ...\circ f_1(x),$$
$$p(\phi_n,E):=\log \sum_{x \in E} exp[\sum_{i=1}^n \phi(f_i \circ ...\circ f_1(x))].$$
It is clear that  $\delta_n \in \mathcal{W}$ and $\phi_n \in C(X).$ A finite set $E \subset X$ is called:

1) $(n,\delta)_M-${\it separated} if for any distinct points $x,y \in E$ we have $(x,y)\notin \delta_n,$

2) $(n,\delta)_M-${\it spanning} if for any $x \in X$ there exists $y \in E$ such that $(x,y) \in \delta_n.$

Let
$$P_{n,\delta}(f_{1,\infty},\phi):=\sup \{p(\phi_n, E): ~E \subset X~and~E~is~(n,\delta)_M-separated\},$$

 $$P_{\delta}(f_{1,\infty},\phi):=\limsup_{n \to \infty} \frac{1}{n} P_{n,\delta}(f_{1,\infty},\phi).$$
\begin{definition}\rm{The quantity
$P_{Mis}(f_{1,\infty},\phi):=\sup_{\delta \in  \mathcal{W}} P_{\delta}(f_{1,\infty},\phi)$
 is called the {\it  Misiurewicz pressure  } of  $f_{1,\infty}=\{f_n: X \to X\}_{n \in \mathbb{N}},$ 
with respect to a potential $\varphi \in C(X)$. Taking the potential $\varphi_0 \equiv 0$ we can define 
Misiurewicz entropy of $f_{1,\infty}$  as
$$h_{Mis}(f_{1,\infty}):=P_{Mis}(f_{1,\infty},\varphi_0).$$}
\end{definition} 

\subsection{Pressure function} 
\begin{proposition}\rm{ The map $\varphi   \longmapsto P_{Mis}(f_{1,\infty}, \varphi)$ is a  pressure function.}
\end{proposition}

\begin{proof}
The proof is based on the three lemmas below.
\end{proof} 
 
\begin{lemma}\rm{If $\varphi,\psi \in C(X)$ and $\varphi \leq\psi$, then $P_{Mis}(f_{1,\infty},\varphi)\leq  P_{Mis}(f_{1,\infty}, \psi).$}
\end{lemma}

 \begin{proof} Choose a neighborhood $\delta \in \mathcal{W}$ and an $(n,\delta)_{M}$-separated set $E.$ Then  for $x \in E$ and potentials $\varphi,\psi \in C(X)$ with $\varphi \leq\psi$ we get the following inequalities 
 $$\sum_{i=1}^n \varphi[f_i \circ ...\circ f_1(x)] \le \sum_{i=1}^n \psi[f_i \circ ...\circ f_1(x)],$$
 $$\sum_{x \in E}\exp \sum_{i=1}^n \varphi[f_i \circ ...\circ f_1(x)] \le \sum_{x \in E}\exp \sum_{i=1}^n \psi[f_i \circ ...\circ f_1(x)],$$
  $$\log \sum_{x \in E}\exp \sum_{i=1}^n \varphi[f_i \circ ...\circ f_1(x)] \le \log \sum_{x \in E}\exp \sum_{i=1}^n \psi[f_i \circ ...\circ f_1(x)],$$
 $$\sup_{E} \{\log \sum_{x \in E}\exp \sum_{i=1}^n \varphi[f_i \circ ...\circ f_1(x)] \}\le \sup_{E} \{\log \sum_{x \in E}\exp \sum_{i=1}^n \psi[f_i \circ ...\circ f_1(x)]\}.$$
 In other words, we obtained the inequality $P_{n,\delta}(f_{1,\infty},\varphi)\le P_{n,\delta}(f_{1,\infty},\psi),$ which yields
 $$P_{\delta}(f_{1,\infty},\varphi)=\limsup_{n \to \infty}\frac{1}{n}P_{n,\delta}(f_{1,\infty},\varphi)\le \limsup_{n \to \infty}\frac{1}{n}P_{\delta}(f_{1,\infty},\psi)=P_{\delta}(f_{1,\infty},\psi).$$ 
 Applying the supremum over $\delta \in \mathcal{W}$ to both sides of the inequality we get
 $$P_{Mis}(f_{1,\infty},\varphi)=\sup_{\delta \in \mathcal{W}}P_{\delta}(f_{1,\infty},\varphi) \le \sup_{\delta \in \mathcal{W}}P_{\delta}(f_{1,\infty},\psi)=P_{Mis}(f_{1,\infty}, \psi).$$
 The proof of the lemma is complete.
 
 \end{proof} 

\begin{lemma}\rm{For any $\varphi \in C(X)$ and $c\in \mathbb R$ the following equality holds $$P_{Mis}(f_{1,\infty},\varphi+c)=  P_{Mis}(f_{1,\infty}, \varphi)+c.$$}
\end{lemma} 

\begin{proof} Fix a potential $\varphi \in C(X).$ For a neighborhood $\delta \in \mathcal{W},$  an $(n,\delta)_{M}$-separated set $E,$ and $c \in \mathbb{R}$ we get
$$\log \sum_{x \in E}\exp \sum_{i=1}^n (\varphi+c)[f_i \circ ...\circ f_1(x)] =\log \sum_{x \in E}\exp\{n\cdot c+ \sum_{i=1}^n \varphi[f_i \circ ...\circ f_1(x)]\} =$$
$$\log \sum_{x \in E}\exp\{n\cdot c\}\cdot \exp \{ \sum_{i=1}^n \varphi[f_i \circ ...\circ f_1(x)]\} =n\cdot c+ \log \sum_{x \in E}\exp\{ \sum_{i=1}^n \varphi[f_i \circ ...\circ f_1(x)]\},$$
therefore
$$\limsup_{n \to \infty}\frac{1}{n}\log \sum_{x \in E}\exp \sum_{i=1}^n (\varphi+c)[f_i \circ ...\circ f_1(x)] =c + \limsup_{n \to \infty}\frac{1}{n}\log \sum_{x \in E}\exp \sum_{i=1}^n \varphi[f_i \circ ...\circ f_1(x)]$$
Applying the supremum over $\delta \in \mathcal{W}$ to both sides of the above equality we get
$$P_{Mis}(f_{1,\infty},\varphi+c)=  P_{Mis}(f_{1,\infty}, \varphi)+c.$$
\end{proof}

\begin{lemma}\rm{For any $\varphi, \psi \in C(X)$ and $a\in
[0,1]$ we have $$P_{Mis}(f_{1,\infty},a\varphi+(1-a)\psi)\le  a \cdot P_{Mis}(f_{1,\infty}, \varphi)+(1-a)\cdot P_{Mis}(f_{1,\infty}, \psi).$$}
\end{lemma} 

\begin{proof} Let  $\varphi, \psi \in C(X)$ and $a\in[0,1].$ 
For a neighborhood $\delta \in \mathcal{W},$  an $(n,\delta)_{M}$-separated set $A \subset X$ we get respectively 
$$\sum_{i=1}^n (a \cdot \varphi+(1-a)\cdot \psi)[f_i \circ ...\circ f_1(x)] =a \cdot \sum_{i=1}^n \varphi[f_i \circ ...\circ f_1(x)] +(1-a)\sum_{i=1}^n \psi[f_i \circ ...\circ f_1(x)],$$

$$\sum_{x \in A} e^{ \sum_{i=1}^n (a \cdot \varphi+(1-a)\cdot \psi)[f_i \circ ...\circ f_1(x)]} =\sum_{x \in A} e^{ a \cdot \sum_{i=1}^n  \varphi[f_i \circ ...\circ f_1(x)]  
  +(1-a) \sum_{i=1}^n\psi[f_i \circ ...\circ f_1(x)]} \le$$
$$ \{\sum_{x \in A} e^{ \sum_{i=1}^n (\varphi[f_i \circ ...\circ f_1(x)] }\}^a \cdot \{\sum_{x \in A} e^{ \sum_{i=1}^n (\psi[f_i \circ ...\circ f_1(x)]} \}^{1-a},               $$
where the last inequality is due to Holder's inequality. Thus
$$\log \sum_{x \in A} e^{ \sum_{i=1}^n (a \cdot \varphi+(1-a)\cdot \psi)[f_i \circ ...\circ f_1(x)]} \le a \log \sum_{x \in A} e^{ \sum_{i=1}^n (\varphi[f_i \circ ...\circ f_1(x)] } +(1-a) \log \sum_{x \in A} e^{ \sum_{i=1}^n (\varphi[f_i \circ ...\circ f_1(x)] }.$$
Taking supremum of both sides over the sets $A$ which are $(n,\delta)_M-$separated and applying the definition of  $P_{n,\delta}(f_{1,\infty},\phi)$ we obtain
$$P_{n,\delta}(f_{1,\infty},a\cdot \varphi +(1-a)\cdot \psi) \le a \cdot P_{n,\delta}(f_{1,\infty},\varphi) +(1-a)\cdot P_{n,\delta}(f_{1,\infty},\psi),$$

$$ \limsup_{n \to \infty}  \frac{1}{n}P_{n,\delta}(f_{1,\infty},a\cdot \varphi +(1-a)\cdot \psi) \le a \limsup_{n \to \infty}\frac{1}{n}\cdot P_{n,\delta}(f_{1,\infty},\varphi) +(1-a)\limsup_{n \to \infty}\cdot
\frac{1}{n} P_{n,\delta}(f_{1,\infty},\psi).$$
In this way we have obtained the inequalities
$$P_{\delta}(f_{1,\infty},a\cdot \varphi +(1-a)\cdot \psi) \le a \cdot P_{\delta}(f_{1,\infty},\varphi) +(1-a)\cdot P_{\delta}(f_{1,\infty},\psi),$$
$$\sup_{\delta \in  \mathcal{W}} P_{\delta}(f_{1,\infty},a\cdot \varphi +(1-a)\cdot \psi) \le a \cdot \sup_{\delta \in  \mathcal{W}} P_{\delta}(f_{1,\infty},\varphi) +(1-a)\cdot \sup_{\delta \in  \mathcal{W}} P_{\delta}(f_{1,\infty},\psi).$$

The proof of the lemma is complete.
\end{proof}

\subsection{Proof of Theorem C} 
\begin{proof} Since $ P_{Mis}(f_{1,\infty}, \cdot) : C(X) \to \mathbb{R}$ is a pressure function
the unifying statement of Theorem~\ref{thm:main-lemma} provides an abstract variational principle 
for the nonautonomous discrete dynamical system determined by $f_{1,\infty}= \{ f_n:X \to X \}^{\infty}_{n=1}$  and Misiurewicz pressure $P_{Mis}(f_{1,\infty}, \varphi)$
with respect to $\varphi \in C(X).$
\end{proof}

\subsection{Proof of Corollary 2}
\begin{proof} The proof of Corollary 2 is very similar to the proof of Corollary 1.
\end{proof}

\subsection{Proof of Theorem D} 
\begin{proof}
The proof of Theorem D is very similar to the proof of Theorem B. We leave this proof to the reader as a simple exercise.
\end{proof}
\bigskip
\noindent
 {\bf Acknowledgements}.The author is grateful to the anonymous referees  for reading the manuscript very carefully,
 for the apposite comments and valuable suggestions that have helped   to improve the manuscript.
\color{black}


\end{document}